\documentclass[12pt]{amsart}
\pdfoutput=1

\usepackage[utf8]{inputenc}
\usepackage{indentfirst}
\usepackage{amsmath}
\usepackage{amsthm}
\usepackage{amssymb}
\usepackage{amsfonts}
\usepackage{microtype}
\usepackage{fancyhdr}
\usepackage{tikz-cd}
\usetikzlibrary{arrows}
\usetikzlibrary{positioning}
\usepackage{xcolor}
\usepackage[colorlinks]{hyperref}
\usepackage{mathtools}
\usepackage{thmtools}

\usepackage{import}
\usepackage{xifthen}
\usepackage{pdfpages}
\usepackage{transparent}

\usepackage[
  margin=1in,
  includefoot,
  footskip=30pt,
]{geometry}

\usepackage{csquotes}

\usepackage[style=alphabetic, citestyle=alphabetic]{biblatex}
\hypersetup{
  colorlinks = true,
  citecolor = {purple!70!black},
  linkcolor = {red!60!black}
}

\renewcommand{\O}{\mathcal{O}}
\renewcommand{\P}{\mathbb{P}}

\newcommand{\Z}{\mathbb{Z}}

\newcommand{\Ext}{\mathrm{Ext}}

\newcommand{\Hom}{\mathrm{Hom}}

\newcommand{\RGamma}{\mathrm{R}\Gamma\,}

\newcommand{\HH}{\mathrm{HH}}

\newcommand{\Dbcoh}{D^b_{\!\mathrm{coh}}}

\newcommand{\mA}{\mathcal{A}}

\newcommand{\dual}{{\scriptstyle\vee}}
\newcommand{\iso}{\simeq}
\newcommand{\caniso}{\cong}
\newcommand{\isoarrow}{\xrightarrow{\sim}}
\newcommand{\monoarrow}{\hookrightarrow}
\newcommand{\epiarrow}{\twoheadrightarrow}

\declaretheoremstyle[
headformat=\NUMBER.\,\NAME\NOTE,
postheadspace=.5em,
spaceabove=6pt,
headfont=\normalfont\small\scshape,
notefont=\normalfont\small\mdseries, notebraces={(}{)},
bodyfont=\normalfont\itshape
]{plainswap}
\declaretheoremstyle[
headformat=\NAME\NOTE,
postheadspace=.5em,
spaceabove=6pt,
headfont=\normalfont\small\scshape,
notefont=\normalfont\small\mdseries, notebraces={(}{)},
bodyfont=\normalfont\itshape
]{nonumplainswap}
\declaretheoremstyle[
headformat=\NUMBER.\,\NAME\NOTE,
postheadspace=.5em,
spaceabove=6pt,
headfont=\normalfont\small\scshape,
notefont=\normalfont\mdseries, notebraces={(}{)},
bodyfont=\normalfont
]{definitionswap}
\declaretheoremstyle[
headformat=\NAME\NOTE,
postheadspace=.5em,
spaceabove=6pt,
headfont=\normalfont\itshape,
notefont=\mdseries, notebraces={(}{)},
bodyfont=\normalfont
]{myremark}

\declaretheorem[style=plainswap, name=Theorem, sharenumber=subsection]{theorem}

\declaretheorem[style=plainswap, numberlike=theorem, name=Proposition]{proposition}

\declaretheorem[style=plainswap, numberlike=theorem, name=Lemma]{lemma}
\declaretheorem[style=plainswap, numberlike=theorem, name=Corollary]{corollary}

\theoremstyle{definition}
\declaretheorem[style=definitionswap, numberlike=theorem, name=Definition]{definition}

\theoremstyle{myremark}
\newtheorem*{remark}{Remark}

\theoremstyle{remark}

\numberwithin{equation}{theorem}

\begin{document}

\title{Categorical Torelli theorem for hypersurfaces}
\author{Dmitrii Pirozhkov}

\begin{abstract}
  Let $X \subset \P^{n+1}$ be a smooth Fano hypersurface of dimension $n$ and degree~$d$. The derived category of coherent sheaves on~$X$ contains an interesting subcategory called the Kuznetsov component~$\mA_X$. We show that this subcategory, together with a certain autoequivalence called the rotation functor, determines $X$ uniquely if~$d > 3$ or if~$d = 3$ and~$n > 3$. This generalizes a result by D.~Huybrechts and J.~Rennemo, who proved the same statement under the additional assumption that $d$ divides $n+2$.
\end{abstract}

\maketitle

\section{Introduction}

Reconstruction of an algebraic variety from some of its invariants is a classical endeavor that started with Torelli proving that a smooth projective curve is determined by its Jacobian as a polarized abelian variety. In 1997 A.~Bondal and D.~Orlov proved \cite{bondal-orlov} that Fano varieties and varieties with ample canonical class can be reconstructed from their derived categories of coherent sheaves. In this note we are interested in some developments arising from Bondal--Orlov's theorem.

To discuss our setting, we first need some notation. Let $k$ be a field of characteristic zero, and let~$V$ be an \mbox{$(n+2)$-dimensional} vector space. We work with an~$n$-dimensional smooth Fano hypersurface~$X \subset \P(V)$ of degree~$d$. The bounded derived category of coherent sheaves on $X$ admits a semiorthogonal decomposition:
\begin{equation}
  \label{eq:intro sod}
  \Dbcoh(X) =  \langle \mA_X, \O_X, \O_X(1), \ldots, \O_X(n-d+1) \rangle,
\end{equation}
where the category $\mA_X$ is called a \emph{Kuznetsov component} or a \emph{residual category} of $\Dbcoh(X)$. The decomposition \eqref{eq:intro sod} implies that, according to \cite[Prop.~3.8]{orlov-glueing}, the category $\Dbcoh(X)$, and hence by Bondal--Orlov's theorem the variety $X$ itself, can be reconstructed from the following pieces of data:
\begin{itemize}
\item The category $\langle \O_X, \O_X(1), \ldots, \O_X(n-d+1) \rangle$ that depends only on $n$ and $d$, not on the choice of the hypersurface $X$;
\item The Kuznetsov component $\mA_X$;
\item A certain glueing data between the two categories above.
\end{itemize}

In general, it is impossible to determine $X$ just from the Kuznetsov component $\mA_X$, without the glueing data. An example with cubic fourfolds is described in \cite{pertusi}. In the paper~\cite{huybrechts-rennemo} D.~Huybrechts and J.~Rennemo suggested a different approach to this problem. They used a particular autoequivalence~$\Phi_{\mA_X}$ of $\mA_X$ called the \emph{rotation functor} (see Definition~\ref{def:rotation functor}). They proved the following reconstruction theorem:

\begin{theorem}[{\cite[Cor.~1.2]{huybrechts-rennemo}}]
  A smooth Fano hypersurface $X \subset \P(V)$ of degree $d$ and dimension $n$ with $d | (n+2)$ is determined by the pair $(\mA_X, \Phi_{\mA_X})$ composed of the Kuznetsov component and the rotation functor (as a dg-category and a dg-endofunctor).
\end{theorem}

In this paper we generalize their approach to any Fano hypersurface, not necessarily satisfying the divisibility condition on $d$ and $n$:

\begin{theorem}[{ = Theorem~\ref{thm:main theorem}}]
  \label{thm:main theorem intro}
  A smooth Fano hypersurface $X \subset \P(V)$ of degree $d$ and dimension $n$ satisfying~$d > 3$ or~$(d = 3, n > 3)$ is determined by the pair $(\mA_X, \Phi_{\mA_X})$ composed of the Kuznetsov component and the rotation functor (as a dg-category and a dg-endofunctor).
\end{theorem}

\begin{remark}
  The case of cubic threefolds, not included in this theorem, has been studied in~\cite{bmms}. For this case the category $\mA_X$ alone suffices to determine $X$.
\end{remark}

The strategy of the proof is similar to the one in \cite{huybrechts-rennemo}. The vector space of natural transformations from the identity functor of $\mA_X$ to the rotation functor is isomorphic to~$V^\dual$~(Lemma~\ref{lem:multiplication on kuznetsov component}).  The~$d$'th power of the rotation functor is a shift-by-two functor on~$\mA_X$~(Theorem~\ref{thm:kuznetsov periodicity}), and thus the~$d$-fold composition of the natural transformations defines a morphism
\begin{equation}
  \label{eq:intro map}
  S^d V^\dual \to \HH^2(\mA_X).
\end{equation}
We show that the kernel of this morphism is exactly the $d$'th graded component of the Jacobian ideal of $X \subset \P(V)$, which determines the hypersurface~$X$ uniquely. A subtle point of the argument is that we do not need an explicit computation of~$\HH^2(\mA_X)$. Instead, we prove that the map~\eqref{eq:intro map} factors through $\HH^2(X)$ and use an easy observation from Lemma~\ref{lem:action is nondegenerate if non-diagonal} that shows that the restriction map $\HH^2(X) \to \HH^2(\mA_X)$ is an injection on a large subspace of $\HH^2(X)$.

There are many classes of varieties which admit semiorthogonal decompositions similar to the one in \eqref{eq:intro sod} in the sense that one of the components of the decomposition is the "interesting" one, and the others are very simple. In these situations one could investigate some refined version of the Bondal--Orlov's theorem. For a review of known results along this direction, see \cite{pertusi-stellari}. We especially remark the Fano threefold case studied in \cite{infinitesimal-torelli}, due to some similarities with our approach.

While preparing the work I learned of an upcoming paper by J.~Rennemo~\cite{rennemo-reconstruction}, who showed that if $d$ does not divide $n+2$ and the pair $(d, n)$ is not of the form $(4, 4k)$, then the Kuznetsov component alone suffices to reconstruct $X$, with no dependency on the rotation functor. I believe that Theorem~\ref{thm:main theorem intro}, though weaker outside of the case $(d = 4, n = 4k)$, is still of interest, in particular due to the uniform handling of all cases.

\textbf{Structure of the paper}. In Section~\ref{sec:jacobian ring} we recall the notion of the Jacobian ring of a hypersurface and its connection with Hochschild cohomology. In Section~\ref{sec:orthogonal to structure sheaf}, following \cite{kuznetsov-v14}, we perform some computations related to the orthogonal to the structure sheaf in the derived category of a hypersurface. We use those results in Section~\ref{sec:kuznetsov component} to study the $d$'th power of the rotation functor on the Kuznetsov component. Finally, in Section~\ref{sec:injectivity on hochschild} we prove the main Theorem~\ref{thm:main theorem intro}.

\textbf{Notation}. Let $K$ be a field of characteristic zero. In this paper all categories are assumed to be triangulated and $K$-linear, all pullbacks and pushforwards are assumed to be derived, and all varieties are assumed to be smooth.

Let $V$ be an $(n+2)$-dimensional vector space over $K$.
We work with a smooth $n$-dimensional Fano hypersurface $X \subset \P(V)$ with $\mathrm{deg}(X) = d$, defined by an equation $f \in S^d V^\dual$. In particular, $d < n+2$.

We use the notation $\P$ for the projective space $\P(V)$. Since we mostly work with objects in the derived categories $\Dbcoh(X)$ and $\Dbcoh(X \times X)$, we sometimes omit the subscript $X$ on objects like the structure sheaf of the diagonal $\O_\Delta \in \Dbcoh(X \times X)$ to avoid the symbol clutter in formulas when the risk of confusion is small.

\textbf{Acknowledgements}. I thank Emanuele Macr\`i for many helpful conversations. I thank Alexander Kuznetsov for advice and suggestions. I also thank Daniel Huybrechts and J\o rgen Rennemo for their explanations concerning Lemma~\ref{lem:multiplication on kuznetsov component}, and for letting me know of \cite{rennemo-reconstruction}.

\section{Hochschild cohomology and Jacobian ring}
\label{sec:jacobian ring}
In \cite{donagi} Ron Donagi proved a Hodge-theoretic Torelli theorem for a (very general) hypersurface satisfying some conditions on the degree and the dimension. The proof relied on the notion of the Jacobian ring of a hypersurface. We also need this notion for the categorical version of Torelli theorem.

\begin{definition}
  \label{def:jacobian ring}
  The \emph{Jacobian ring} of a hypersurface $X \subset \P(V)$ defined by an equation~$f \in S^d V^\dual$ is the graded ring given by the quotient
  \[
    J^\bullet(f) := S^\bullet V^\dual/\langle \tfrac{\partial f}{\partial v} \rangle_{v \in V}.
  \]
  The graded ideal generated by the partial derivatives of $f$ is called the \emph{Jacobian ideal} of $X$.
\end{definition}

What most interests us is a relation between the Jacobian ring and the Hochschild cohomology of the hypersurface. One aspect of this relation is demonstrated in the combination of Lemma~\ref{lem:h1t and jacobian ideal} and Proposition~\ref{prop:hh2 and jacobian ideal}. For general information on Hochschild cohomology see, for example, \cite{kuznetsov-oldhochschild}. We only recall that Hochschild cohomology of a variety can be computed by the formula
\[
  \HH^\bullet(X) = \Ext^\bullet_{X \times X}(\O_\Delta, \O_\Delta),
\]
and the Hochschild cohomology of an (admissible) subcategory $\mA_X \subset \Dbcoh(X)$ can be computed by the formula
\begin{equation}
  \label{eq:hochschild definition subcategory}
  \HH^\bullet(\mA_X) = \Ext^\bullet_{X \times X}(P, P),
\end{equation}
where the object~$P \in \Dbcoh(X \times X)$ is a Fourier--Mukai kernel of the (left or right) projection functor to~$\mA_X$.

\begin{lemma}
  \label{lem:h1t and jacobian ideal}
  Consider the normal bundle short exact sequence:
  \[
    0 \to T_X \to T_{\P}|_X \to \O_X(d) \to 0.
  \]
  The connecting homomorphism in the long exact sequence of cohomology groups induces a morphism
  \[
    S^d V^\dual \iso H^0(\P, \O_{\P}(d)) \to H^0(X, \O_X(d)) \to H^1(X, T_X).
  \]
  Then this map is surjective and it identifies $H^1(X, T_X)$ with the $d$'th graded component of the Jacobian ring $J^d(f)$.
\end{lemma}

\begin{proof}
  The long exact sequence of cohomology groups of the normal short exact sequence contains the following fragment:
  \begin{equation}
    \label{eq:cohomology of normal exact sequence}
    H^0(T_{\P}|_X) \to H^0(\O_X(d)) \to H^1(T_X) \to H^1(T_\P|_X).
  \end{equation}
  
  Note that the tangent bundle of $\P$ fits into the Euler short exact sequence
  \[
    0 \to \O_\P \to V \otimes \O_\P(1) \to T_\P \to 0.
  \]
  Since $d < n+2$ it is easy to compute that $H^1(X, T_\P|_X) = 0$ and $H^0(X, T_\P|_X) = V^\dual \otimes V/\langle \mathrm{id}_V \rangle$. Thus the sequence \eqref{eq:cohomology of normal exact sequence} simplifies:
  \[
    V^\dual \otimes V/\langle \mathrm{id}_V \rangle
    \to
    S^d V^\dual / \langle f \rangle
    \to
    H^1(T_X) \to 0.
  \]
  Here the first map sends a decomposable tensor~$\xi \otimes v \in V^\dual \otimes V$  to the element~$\xi \cdot \frac{\partial f}{\partial v} \in S^d V^\dual$. The image of this map in $S^d V^\dual$ is thus the $d$'th component of the Jacobian ideal. Since $f$ lies in its own Jacobian ideal, we conclude that $H^1(X, T_X) \iso J^d(f)$.
\end{proof}

Recall the notion of the \emph{universal Atiyah class} $\mathrm{At} \in \Ext^1(\O_\Delta, \Delta_*\Omega^1_X)$ (see, e.g., \cite{kuznetsov-markushevich}).
Abusing the notation, we denote the following composition also by $\mathrm{At}$:
\[
  H^1(X, T_X) \iso \Ext_X^1(\Omega^1_X, \O_X) \xrightarrow{\Delta_*} \Ext^1_{X \times X}(\Delta_*\Omega^1_X, \O_\Delta) \xrightarrow{- \circ \mathrm{At}} \Ext^2_{X \times X}(\O_\Delta, \O_\Delta) \caniso \HH^2(X).
\]
By the Hochschild--Kostant--Rosenberg theorem (e.g., \cite[Cor.~2.6]{swan-hkr}) this morphism is an injection.

Recall the (universal) linkage class $\epsilon_X$ for a hypersurface $X \subset \P(V)$ \cite[Sec.~3]{kuznetsov-markushevich}: the derived restriction of $\O_{\Delta_\P}$ to $X \times X \subset \P \times \P$ is a complex with two adjacent cohomology sheaves. Thus it fits into a distinguished triangle
\begin{equation}
  \label{eq:universal linkage}
  \O_{\Delta_\P}|_{X \times X} \to \O_\Delta \xrightarrow{\epsilon_X} \O_\Delta(-d)[2].
\end{equation}
The gluing morphism~$\epsilon_X \in \Ext^2_{X \times X}(\O_\Delta, \O_\Delta(-d))$ between the cohomology sheaves is called the \emph{universal linkage class}.

\begin{proposition}
  \label{prop:hh2 and jacobian ideal}
  The composition with the universal linkage class defines a morphism:
  \[
    S^d V^\dual \epiarrow H^0(X, \O_X(d)) \xrightarrow{\Delta_*} \Hom_{X \times X}(\O_\Delta, \O_\Delta(d)) \xrightarrow{\epsilon_X \circ -} \Ext^2_{X \times X}(\O_\Delta, \O_\Delta) \caniso \HH^2(X).
  \]
  such that the following triangle commutes:
  \begin{equation}\begin{tikzcd}
      \label{eq:hh2 and jacobian ideal}
      {S^d V^\dual} & {H^1(X, T_X)} \\
      & {\HH^2(X)}
      \arrow["{\mathrm{At}}", hook', from=1-2, to=2-2]
      \arrow[from=1-1, to=1-2]
      \arrow[from=1-1, to=2-2]
    \end{tikzcd}\end{equation}
  where the horizontal arrow is the map from Lemma~\textup{\ref{lem:h1t and jacobian ideal}}.
\end{proposition}

\begin{proof}
  Let $\nu\colon \Omega^1 \to \O(-d)[1]$ be the extension class of the conormal exact sequence
  \[
    0 \to \O(-d) \to \Omega^1_\P|_X \to \Omega^1_X \to 0.
  \]
  By \cite[Th.~3.2]{kuznetsov-markushevich} the universal linkage class is equal to the composition
  \[
    \O_\Delta \xrightarrow{\mathrm{At}} \Delta_*\Omega^1[1] \xrightarrow{\Delta_*\nu} \O_\Delta(-d)[2]
  \]
  of the universal Atiyah class and the pushforward of $\nu$ along the diagonal.
  Unwinding the definitions, we see that the diagram~\eqref{eq:hh2 and jacobian ideal} commutes on an element $g \in S^d V^\dual$ if and only if the diagram below commutes:
  \[\begin{tikzcd}
      {\O_\Delta} & {\Delta_* \Omega^1[1]} & {\O_\Delta(-d)[2]} \\
      {\O_\Delta(d)} & {\Delta_*\Omega^1(d)[1]} & {\O_\Delta[2]}
      \arrow["{\mathrm{At}}", from=1-1, to=1-2]
      \arrow["{\Delta_*\nu}", from=1-2, to=1-3]
      \arrow["{\mathrm{At}(d)}", from=2-1, to=2-2]
      \arrow["{\Delta_*\nu(d)}", from=2-2, to=2-3]
      \arrow["{\cdot g}", from=1-3, to=2-3]
      \arrow["{\cdot g}", from=1-1, to=2-1]
    \end{tikzcd}\]
  This commutativity is clear since multiplication by $g$ is a natural transformation.
\end{proof}

\section{The orthogonal to the structure sheaf}
\label{sec:orthogonal to structure sheaf}

Since $X \subset \P(V)$ is by assumption a Fano hypersurface, the structure sheaf $\O_X \in \Dbcoh(X)$ is an exceptional object. In this section we study some properties of the right-orthogonal subcategory $\O_X^\perp \subset \Dbcoh(X)$. The goal is to perform some explicit computations involving the projection functor to $\O_X^\perp$. We will rely on them in Section~\ref{sec:kuznetsov component} where we study the Kuznetsov component $\mA_X \subset \Dbcoh(X)$. Nothing in this section is new, all results are already in \cite{kuznetsov-v14}.

We begin with a brief reminder on Fourier--Mukai transforms to fix the notation. For details, see \cite{HuybFM}.
Given an object $K \in \Dbcoh(X \times X)$ we define the Fourier--Mukai functor~$\Phi_K\colon \Dbcoh(X) \to \Dbcoh(X)$ by the formula $\pi_{2 *}(\pi_1^*(-) \otimes K)$.  For any $d \in \Z$ the Fourier--Mukai transform $\Phi_{\O_{\Delta}(d)}$ is the functor given by the twist by $\O_X(d)$. For any pair of objects~$F, G \in \Dbcoh(X)$ the Fourier--Mukai transform along the exterior product $F \boxtimes G$ is the functor
\[
  \Phi_{F \boxtimes G}(-) := \RGamma(- \otimes F) \otimes G.
\]
The \emph{convolution} of two kernels $K_1, K_2 \in \Dbcoh(X \times X)$ is defined by the formula
\[
  K_1 \circ K_2 := \pi_{1 3 *}(\pi_{1 2}^*(K_1) \otimes \pi_{2 3}^*(K_2)),
\]
and it satisfies $\Phi_{K_1 \circ K_2} = \Phi_{K_1} \circ \Phi_{K_2}$. For line bundles of the form~$\O_X(a, b) := \O_X(a) \boxtimes \O_X(b)$ the convolution equals
\begin{equation}
  \label{eq:convolution of line bundles}
  \O_X(a, b) \circ \O_X(a^\prime, b^\prime) \caniso \RGamma(\O_X(b + a^\prime)) \otimes \O_X(a, b^\prime).
\end{equation}

Now we return to the main object of this section.

\begin{definition}
  \label{def:kernels for rotations}
  We define the object $Q_0 \in \Dbcoh(X \times X)$ as the complex
  \[
    Q_0 := [ \O_X \boxtimes \O_X \to \O_\Delta ],
  \]
  in degrees $-1$ and $0$, representing the left projection functor to the subcategory $\O_X^\perp \subset \Dbcoh(X)$. We define the objects $Q_i$ for $i > 0$ recursively:
  \[
    Q_i := Q_{i-1} \circ \O_\Delta(1) \circ Q_0,
  \]
  where the symbol $\circ$ denotes the convolution of Fourier--Mukai kernels in $\Dbcoh(X \times X)$.
\end{definition}

\begin{remark}
  The functor $\Dbcoh(X) \to \Dbcoh(X)$ represented by the object $Q_1 \in \Dbcoh(X \times X)$ is called a \emph{rotation functor} in \cite{kuznetsov-v14}. It can alternatively be described as a composition
  \[
    \Dbcoh(X) \xrightarrow{Q_0} \Dbcoh(X) \xrightarrow{- \otimes \O_X(1)} \Dbcoh(X) \xrightarrow{Q_0} \Dbcoh(X),
  \]
  where the Fourier--Mukai transform along the object~$Q_0$ is the left projection to the subcategory~$\O_X^\perp$.
\end{remark}

\begin{lemma}
  \label{lem:multiplication on the subcategory}
  For any $i \geq 0$ there is a natural morphism $\O_\Delta(i) \to Q_i$ which induces a map
  \[
    m_{Q_i}\colon S^i V^\dual \to \Hom(\O_\Delta, \O_\Delta(i)) \to \Hom(\O_\Delta, Q_i) \to \Hom(Q_0, Q_i).
  \]
  Furthermore, for $i = 1$ the map $m_{Q_1}$ is an isomorphism of vector spaces.
\end{lemma}

\begin{proof}
  By the definition of the left projection functor there is a morphism
  \[
    \O_\Delta \to Q_0.
  \]
  We define the map $\O_\Delta(i) \to Q_i$ by repeatedly using the map $\O_\Delta \to Q_0$ as follows:
  \[\begin{tikzcd}
      {\O_\Delta(i)} & {\O_\Delta \circ \O_\Delta(1) \circ \O_\Delta \circ \cdots \circ \O_\Delta(1) \circ \O_\Delta} \\
      {Q_i} & {Q_0 \circ \O_\Delta(1) \circ Q_0 \circ \cdots \circ \O_\Delta(1) \circ Q_0 }
      \arrow["{=}", from=1-1, to=1-2]
      \arrow["{=}", from=2-1, to=2-2]
      \arrow[from=1-2, to=2-2]
    \end{tikzcd}\]

  The fact that $m_{Q_1}$ is an isomorphism can be checked by a straightforward computation using the resolution
  \[
    Q_1 \caniso [ V^\dual \otimes \O_{X \times X} \to \O(1,0) \oplus \O(0, 1) \to \O_\Delta(1) ]
  \]
  obtained by using Definition~\ref{def:kernels for rotations} and the formula~\eqref{eq:convolution of line bundles}.
\end{proof}

\begin{definition}
  \label{def:beilinson resolution}
  We denote by $B_\Delta$ the Beilinson's resolution of the diagonal on $\P(V) \times \P(V)$:
  \[
    B_\Delta := [ \O_{\P(V)}(-n-1) \boxtimes \Omega_{\P(V)}^{n+1}(n+1) \to \ldots \to \O_{\P(V)}(-1) \boxtimes \Omega_{\P(V)}^1(1) \to \O_{\P(V)} \boxtimes \O_{\P(V)} ].
  \]
  For any $0 \leq i \leq n+1$ we denote by $s_{\geq -i}(B_\Delta)$ the stupid truncation of this resolution:
  \[
    s_{\geq -i}(B_\Delta) := [ \O_{\P(V)}(-i) \boxtimes \Omega_{\P(V)}^{i}(i) \to \ldots \to \O_{\P(V)} \boxtimes \O_{\P(V)} ].
  \]
\end{definition}

\begin{theorem}[{\cite{kuznetsov-v14}}]
  \label{thm:kernels for rotation powers}
  For $0 \leq i < d$ the morphism $\O_\Delta(i) \to Q_i$ from Lemma~\textup{\ref{lem:multiplication on the subcategory}} fits into an exact triangle
  \[
    s_{\geq -i}(B_\Delta)|_{X \times X} \otimes (\O_X(i) \boxtimes \O_X) \xrightarrow{\psi_i} \O_\Delta(i) \to Q_i,
  \]
  and the map~$\psi_i$ is a twist by~$\O_X(i) \boxtimes \O_X$ of the composition
  \[
    s_{\geq -i}(B_\Delta)|_{X \times X} \to B_\Delta|_{X \times X} \isoarrow (\O_{\Delta_{\P(V)}})|_{X \times X} \to \O_\Delta.
  \]
\end{theorem}

\begin{proof}
  The statement is implicitly contained in the proof of \cite[Lem.~4.2]{kuznetsov-v14}, and can be proved by an inductive computation. For the sake of clarity, we sketch the argument. The base case $i = 0$ is true by definition of $Q_0$. Suppose the statement holds for $Q_i$ with $i < d-1$ and we want to prove it for the object
  \[
  Q_{i+1} \caniso Q_i \circ \O_\Delta(1) \circ [ \O_X(0 ,0) \to \O_{\Delta} ] \caniso \mathrm{Cone}(Q_i \circ \O_X(0, 1) \to Q_i \circ \O_\Delta(1)).
  \]
  This description as a cone, together with the formula for $Q_i$ that we know by induction, shows that $Q_{i+1}$ can be represented by the following complex:

  \begin{equation}\begin{tikzcd}
      \label{eq:inductive q}
	{\O_X \boxtimes (H^0(\O_X(1)) \otimes \Omega^i_{\P}(i)|_X)} &[-6mm] \ldots &[-5mm] {\O_X \boxtimes (H^0(\O_X(i+1)) \otimes \O_X)} &[-6mm] {\O_X \boxtimes \O_X(i+1)} \\
	{\O_X(1) \boxtimes \Omega^i_\P(i)|_X} & \ldots & {\O_X(i+1) \boxtimes \O_X} & {\O_\Delta(i+1)}
	\arrow[from=1-1, to=1-2]
	\arrow[from=2-1, to=2-2]
	\arrow[from=1-2, to=1-3]
	\arrow[from=2-2, to=2-3]
	\arrow[from=1-3, to=1-4]
	\arrow[from=2-3, to=2-4]
	\arrow[from=1-1, to=2-1]
	\arrow[from=1-2, to=2-2]
	\arrow[from=1-3, to=2-3]
	\arrow[from=1-4, to=2-4]
      \end{tikzcd}
    \end{equation}
  By induction hypothesis the differentials in the bottom row are given by the contraction with the restriction of the tautological section of $\O_{\P(V)}(-1) \boxtimes T_{\P(V)}(1)$, and the differentials in the top row are induced from that section.
  Note that $H^0(\O_X(j)) \caniso S^j V^\dual$ for any $j < d$.
  Recall that on the projective space we have a resolution for $\Omega^{i+1}_\P(i+1)$ given by a Koszul complex:
  \[
    0 \to \Omega_\P^{i+1}(i+1) \to V^\dual \otimes \Omega_P^i(i) \to \ldots \to \O \otimes S^{i+1} V^\dual \to \O_\P(i+1) \to 0.
  \]
  Thus, if $i+1 < d$,  we recognize the upper row in the diagram \eqref{eq:inductive q} to be a complex quasiisomorphic to a single vector bundle $\O_X \boxtimes \Omega^{i+1}_\P(i+1)$, put to the leftmost degree. This finishes the inductive argument.
\end{proof}

\begin{theorem}[{\cite{kuznetsov-v14}}]
  \label{thm:rotation periodicity}
  There exists a morphism $\varphi_{Q_d}\colon Q_d \to Q_0[2]$ such that the following diagram commutes:
  \begin{equation}
    \label{eq:rotation periodicity commutes}
    \begin{tikzcd}
      {S^d V^\dual} & {\Hom(Q_0, Q_d)} \\
      {\HH^2(X)} & {\Ext^2(Q_0, Q_0)}
      \arrow["{\varphi_{Q_d} \circ -}", from=1-2, to=2-2]
      \arrow["{m_{Q_d}}", from=1-1, to=1-2]
      \arrow[from=1-1, to=2-1]
      \arrow[from=2-1, to=2-2]
    \end{tikzcd}
  \end{equation}
  Here the left vertical map is from Proposition~\textup{\ref{prop:hh2 and jacobian ideal}}, the top horizontal map is from Lemma~\textup{\ref{lem:multiplication on the subcategory}}, and the bottom horizontal arrow comes from the identification of $\Ext^2(Q_0, Q_0)$ with $\HH^2(\O_X^\perp)$.
\end{theorem}
\begin{proof}
  This is, again, implicitly contained in the proof of \cite[Lem.~4.2]{kuznetsov-v14}. To explain the commutativity of the diagram, we repeat the argument.

  Denote by $\widetilde{Q_d}$ the cone
  \begin{equation}
    \label{eq:widetilde qd}
    \widetilde{Q_d} := \mathrm{Cone}(s_{\geq -d}(B_\Delta|_{X \times X}) \otimes \O(d, 0) \to \O_\Delta(d))
  \end{equation}
  as in Theorem~\ref{thm:kernels for rotation powers}. The diagram~\eqref{eq:inductive q} for~$i = d-1$ shows that the difference between~$\widetilde{Q_d}$ and~$Q_d$ arises from the fact that $H^0(X, \O_X(d))$ is isomorphic not to~$S^d V^\dual$, but to the quotient~$S^d V^\dual / \langle f \rangle$, where~$f$ is the equation of the hypersurface~$X \subset \P(V)$. More precisely, there is a distinguished triangle
  \[
    \langle f \rangle \cdot \O_X \boxtimes \O_X[2] \to \widetilde{Q_d} \to Q_d.
  \]
  Note that the convolution $(\O_X \boxtimes \O_X) \circ Q_0$ is a zero object since the Fourier--Mukai transform along $\O_X \boxtimes \O_X$ vanishes on $\O_X^\perp$, and $Q_0$ is exactly the projector to $\O_X^\perp$. Hence the convolution on the right with $Q_0$ produces an isomorphism
  \[
    \widetilde{Q_d} \circ Q_0 \to Q_d \circ Q_0 = Q_d.
  \]

  Consider now the commutative diagram of distinguished triangles arising from the stupid truncation of $B_\Delta|_{X \times X} \iso \O_{\Delta_\P}|_{X \times X}$:

  \begin{equation}
    \label{eq:rotation periodicity comparison}
    \begin{tikzcd}
	{s_{\geq -d}(B_\Delta|_{X \times X}) \otimes \O(d, 0)} & {\O_\Delta(d)} & {\widetilde{Q_d}} \\
	{B_\Delta|_{X \times X} \otimes \O(d, 0)} & {\O_\Delta(d)} & {\O_\Delta[2]}
	\arrow[from=1-1, to=1-2]
	\arrow[from=1-2, to=1-3]
	\arrow[from=1-3, to=2-3]
	\arrow["", from=2-1, to=2-2]
	\arrow["{\epsilon_X}", from=2-2, to=2-3]
	\arrow["{=}", from=1-2, to=2-2]
	\arrow[from=1-1, to=2-1]
      \end{tikzcd}
    \end{equation}
  Since $B_\Delta$ is a resolution of the structure sheaf of the diagonal $\O_{\Delta_{\P}} \in \Dbcoh(\P \times \P)$, the bottom horizontal triangle is the universal linkage class as defined in \eqref{eq:universal linkage}.

  Using the rightmost vertical map, we define the morphism $Q_d \to Q_0[2]$ as the composition:
  \[
    Q_d \iso \widetilde{Q_d} \circ Q_0 \to \O_\Delta[2] \circ Q_0 \caniso Q_0[2].
  \]
  It only remains to show the commutativity of the diagram~\eqref{eq:rotation periodicity commutes}.
  Recall the natural morphism~$\O_\Delta \to Q_d$ defined in Lemma~\ref{lem:multiplication on the subcategory}. It is easy to see from the definition~\eqref{eq:widetilde qd} that this morphism lifts to the map $\O_\Delta(d) \to \widetilde{Q_d}$, and thus the map $m_{Q_d}$ factors through the map
  \[
    S^d V^\dual \to \Hom(\O_\Delta, \O_\Delta(d)) \to \Hom(\O_\Delta, \widetilde{Q_d}).
  \]
  Since the rightmost square in the diagram~\eqref{eq:rotation periodicity comparison} commutes, we additionally see that the map $m_{Q_d}$ factors through the composition with the universal linkage class:
  \[
    S^d V^\dual \to \Hom(\O_\Delta, \O_\Delta(d)) \xrightarrow{\epsilon_X \circ -} \Hom(\O_\Delta, \O_\Delta[2]),
  \]
  and the commutativity of the diagram \eqref{eq:rotation periodicity commutes} follows from Proposition~\ref{prop:hh2 and jacobian ideal}.
\end{proof}

\section{Kuznetsov components}
\label{sec:kuznetsov component}

We begin by discussing the basic properties of Kuznetsov components and their rotation functors.

\begin{definition}
  \label{def:kuznetsov component}
  The \emph{Kuznetsov component} of the hypersurface $X \subset \P(V)$ of degree $d < n+2$ is the category $\mA_X$ defined as the left orthogonal to the exceptional sequence
  \[
    \langle \O_X, \O_X(1), \ldots, \O_X(n-d+1) \rangle
  \]
  in the category $\Dbcoh(X)$.
\end{definition}

\begin{definition}
  \label{def:functors p}
  We define the object $P_0 \in \Dbcoh(X \times X)$ to be the Fourier--Mukai kernel of the left projector to the subcategory $\mA_X \subset \Dbcoh(X)$. Recall the fundamental triangle of projector objects:
  \begin{equation}
    \label{eq:kuznetsov projectors triangle}
    P_0^\prime \to \O_\Delta \to P_0.
  \end{equation}
  where $P_0^\prime$ is the right projector to the subcategory $\langle \O_X, \ldots, \O_X(n-d+1)\rangle$.
  
  We define the objects $P_i$ for $i > 0$ recursively:
  \[
    P_i := P_{i-1} \circ \O_\Delta(1) \circ P_0,
  \]
  where the symbol $\circ$ denotes the convolution of Fourier--Mukai kernels in $\Dbcoh(X \times X)$.
\end{definition}

\begin{definition}
  \label{def:rotation functor}
  The \emph{rotation functor} $\Phi_{\mA_X}\colon \mA_X \to \mA_X$ of $\mA_X$ is defined as the composition
  \[
    \mA_X \monoarrow \Dbcoh(X) \xrightarrow{- \otimes \O_X(1)} \Dbcoh(X) \epiarrow \mA_X.
  \]
  It can alternatively be described as the Fourier--Mukai transform along the kernel $P_1$.
\end{definition}

The following result by Huybrechts and Rennemo computes the space of natural transformations from the identity functor on $\mA_X$ to the rotation functor.

\begin{lemma}[{\cite[Lem.~3.1]{huybrechts-rennemo}}]
  \label{lem:multiplication on kuznetsov component}
  For any $i \geq 0$ there is a natural morphism $\O_\Delta(i) \to P_i$, which induces a map
  \[
    m_{P_i}\colon S^i V^\dual \to \Hom(\O_\Delta, \O_\Delta(i)) \to \Hom(\O_\Delta, P_i) \to \Hom(P_0, P_i).
  \]
  Furthermore, if $d > 3$ or if $d = 3$ and $n > 3$, the map $m_{P_1}\colon V^\dual \to \Hom(P_0, P_1)$  is an isomorphism of vector spaces.
\end{lemma}

For the proof of the main Theorem~\ref{thm:main theorem intro} we need some information about natural transformations from the identity functor of $\mA_X$ to the $d$'th power of the rotation functor. The following two results are sufficient for our purposes.

\begin{lemma}
  \label{lem:compatibility for rotations}
  There exists a natural morphism $P_0 \to Q_0$, which induces a map $P_i \to Q_i$ for any $i \geq 0$. The precomposition with $P_0$ transforms this map into an isomorphism:
  \[
    P_i \caniso P_i \circ P_0 \isoarrow Q_i \circ P_0.
  \]
\end{lemma}

\begin{remark}
  Since $P_0$ is the projector to the subcategory $\mA_X$, the last claim of Lemma~\ref{lem:compatibility for rotations} essentially means that the Fourier--Mukai transforms $\Dbcoh(X) \to \Dbcoh(X)$ along the two kernels~$P_i$ and~$Q_i$ agree on the subcategory $\mA_X \subset \Dbcoh(X)$.
\end{remark}

\begin{proof}
  Note that $P_0$ and $Q_0$ are projector functors to subcategories $\mA_X$ and $\O_X^\perp$, respectively. Since $\mA_X \subset \O_X^\perp$, the statement is true for $i = 0$.
  Since $P_i$ and $Q_i$ are defined inductively in terms of $P_0$ and $Q_0$, it is enough to show that the convolution with the map $P_0 \to Q_0$ induces an isomorphism
  \[
    P_0 \circ \O_\Delta(1) \circ P_0 \isoarrow Q_0 \circ \O_\Delta(1) \circ P_0.
  \]
  Since $Q_0$ and $P_0$ are projectors to the subcategories $\O_X^\perp$ and $\mA_X = \langle \O_X, \ldots, \O_X(n-d+1)\rangle^\perp$, respectively, it is enough to show that the image of the Fourier--Mukai transform along the object $\O_\Delta(1) \circ P_0$ lies in the orthogonal to the exceptional sequence $\langle \O_X(1), \ldots, \O_X(n-d+1)\rangle$. This is clear since $P_0$ is the projector to $\mA_X$.
\end{proof}

\begin{theorem}[{\cite{kuznetsov-v14}}]
  \label{thm:kuznetsov periodicity}
  There is an isomorphism $\varphi_{P_d}\colon P_d \iso P_0[2]$ such that the following diagram commutes:
  \begin{equation}
    \label{eq:kuznetsov periodicity}
    \begin{tikzcd}
      {S^d V^\dual} & {\Hom(P_0, P_d)} \\
      {\HH^2(X)} & {\Ext^2(P_0, P_0)}
      \arrow["{\varphi_{P_d} \circ -}", from=1-2, to=2-2]
      \arrow["{m_{P_d}}", from=1-1, to=1-2]
      \arrow[from=1-1, to=2-1]
      \arrow[from=2-1, to=2-2]
    \end{tikzcd}
  \end{equation}
  Here the left vertical map is from Proposition~\textup{\ref{prop:hh2 and jacobian ideal}}, and the bottom horizontal arrow comes from the identification of $\Ext^2(P_0, P_0)$ with $\HH^2(\mA_X)$.
\end{theorem}

\begin{proof}
  This is proved in \cite[Lem.~4.2]{kuznetsov-v14}. We repeat the argument for the convenience of the reader.
  By Lemma~\ref{lem:compatibility for rotations} the convolution $Q_d \circ P_0$ is naturally isomorphic to $P_d$.
  Recall the morphism $\varphi_{Q_d}\colon Q_d \to Q_0[2]$ from Theorem~\ref{thm:rotation periodicity}. Let $\varphi_{P_d}\colon P_d \to P_0[2]$ be the convolution~$\varphi_{Q_d} \circ P_0$. Then by Theorem~\ref{thm:rotation periodicity} the following diagram commutes

  \[\begin{tikzcd}
	{S^d V^\dual} & {\Hom(Q_0, Q_d)} & {\Hom(P_0, P_d)} \\
	{\HH^2(X)} & {\Ext^2(Q_0, Q_0)} & {\Ext^2(P_0, P_0)}
	\arrow["{\varphi_{Q_d} \circ -}", from=1-2, to=2-2]
	\arrow["{- \circ P_0}", from=1-2, to=1-3]
	\arrow["{- \circ P_0}", from=2-2, to=2-3]
	\arrow["{\varphi_{P_d} \circ -}", from=1-3, to=2-3]
	\arrow[from=1-1, to=2-1]
	\arrow[from=1-1, to=1-2]
	\arrow[from=2-1, to=2-2]
      \end{tikzcd}\]

    The composition of the maps in the upper row is equal to $m_{P_d}$ by definition. Thus the commutativity of the diagram~\eqref{eq:kuznetsov periodicity} is proved. It remains only to show that the map $\varphi_{P_d}$ is an isomorphism. To do this, consider the convolution of the diagram~\eqref{eq:rotation periodicity comparison} used in the proof of Theorem~\ref{thm:rotation periodicity} with the object~$P_0$. Note that $\widetilde{Q_d} \circ P_0 \caniso Q_d \circ P_0 \caniso P_d$. Thus the rightmost vertical map is exactly the morphism $\varphi_{P_d}\colon P_d \to P_0[2]$, and its cone is isomorphic to the object
    \begin{equation}
      \label{eq:cone of the isomorphism}
      (s_{\leq -d-1}(B_\Delta|_{X \times X}) \otimes \O(d, 0)) \circ P_0[1].
    \end{equation}
    To show that this cone is zero, by Definition~\ref{def:beilinson resolution} it is enough to check that the convolution
    \[
      (\O_X(-k + d) \boxtimes \Omega_\P^k(k)|_X) \circ P_0
    \]
    vanishes for any $k$ satisfying $n+1 \geq k \geq d+1$. Since $P_0$ is the projector to the Kuznetsov component $\mA_X$, any object in the image of $P_0$ is right-orthogonal to $\O_X(-k+d)$ for $k \in [d+1; n+1]$, and hence the object~\eqref{eq:cone of the isomorphism} is zero.
\end{proof}

\section{Rotation functors and Hochschild cohomology}
\label{sec:injectivity on hochschild}

Recall that we work with a Fano hypersurface $X \subset \P(V)$ of dimension $n$ and degree $d$. In particular, $d < n+2$.

\begin{lemma}
  \label{lem:nondiagonal diamonds for hypersurfaces}
  If $d > 3$ or if $d = 3$ and $n \geq 3$, then the Hodge diamond of $X$ is not diagonal.
\end{lemma}

\begin{proof}
  By Griffiths' theorem \cite[Th.~8.3]{griffiths} for any $0 \leq p \leq n$ there exists an isomorphism
  \[
    H^{p, n - p}_{\mathrm{prim}} \iso J^{t_p}(f)
  \]
  between the primitive part of the cohomology of $X$ and a particular graded component of the Jacobian ring, where $t_p = (n - p + 1)d - (n+2)$. Since $X$ is smooth, the Jacobian ring is a finite-dimensional graded ring such that any graded component in degrees between~$0$ and~$(d-2)(n+2)$ is nonzero (see, e.g., \cite{donagi}). Thus it is enough to find some $p \neq n/2$ such that $t_p$ lies between $0$ and $(d-2)(n+2)$.

  If~$n = 2m+1$ is an odd number, the condition $d \geq 3$ implies that we can take $p = m$, so we get that~$H^{m, m+1}(X) \neq 0$. If~$n = 2m$ is an even number, the condition~$d \geq 3, n \geq 4$ implies that we can take~$p = m-1$, i.e.,~$H^{m-1, m+1}(X) \neq 0$. 
\end{proof}

\begin{lemma}
  \label{lem:action is nondegenerate if non-diagonal}
  Let $X \subset \P(V)$ be a hypersurface. Let $\mA_X \subset \Dbcoh(X)$ be an admissible subcategory. Assume that the orthogonal subcategory~$\mA_X^\perp \subset \Dbcoh(X)$ has a full exceptional collection and the Hodge diamond of $X$ is not diagonal. Then the composition
  \[
    H^1(X, T_X) \monoarrow \HH^2(X) \to \HH^2(\mA_X)
  \]
  is an injection.
\end{lemma}

\begin{remark}
  A similar idea in a different situation has been recently used in \cite[Thm.~1.2]{infinitesimal-torelli}.
\end{remark}

\begin{proof}
  Since $\mA_X^\perp$ is generated by an exceptional collection, by the additivity of Hochschild homology \cite[Cor.~7.5]{kuznetsov-oldhochschild} we have
  \begin{equation}
    \label{eq:hochschild additivity}
    \HH_*(X) \caniso \HH_*(\mA_X) \oplus \HH_*(\mathrm{pt})^{\oplus k},
  \end{equation}
  where $k$ is the length of the exceptional collection in $\mA_X^\perp$. In particular, for any $i \neq 0$ we have an equality~$\HH_i(X) = \HH_i(\mA_X)$. There is an action of Hochschild cohomology on Hochschild homology, and this action is compatible with the decomposition~\eqref{eq:hochschild additivity} by construction (see, e.g., \cite[Prop.~6.1]{addington-thomas} for the case where $\mA_X$ is isomorphic to $\Dbcoh(S)$ for some smooth projective variety $S$; the proof works in general).

  Let $\xi \in H^1(X, T_X)$ be a nonzero element. We want to show that its image in $\HH^2(\mA_X)$ is nonzero. It is enough to show that the class of $\xi$ in $\HH^2(X)$ acts on $\HH_*(\mA_X)$ nontrivially. By~\eqref{eq:hochschild additivity} it is enough to show that $\xi$ acts nontrivially on the non-zero degree part of the Hochschild homology of $X$, i.e., to find a complementary class~$\widetilde{\xi} \in H^1(X, T_X)^{\otimes N-1}$ and some integer~$a$ so that the action map
  \begin{equation}
    \label{eq:action on hochschild cohomology}
    \HH_a(X) \xrightarrow{\xi \cdot \widetilde{\xi} \cdot -} \HH_{a+2N}(X)
  \end{equation}
  is nontrivial, $a \neq 0$ and $a+2N \neq 0$.

  Using the (Todd-twisted) Hochschild--Kostant--Rosenberg isomorphism~\cite[Th.~1.4]{hkr-action-compatibility} the action of Hochschild cohomology on Hochschild homology can be reinterpreted in terms of the Hodge structure. Namely, under the isomorphisms
  \[
    \HH_a(X) \iso \bigoplus_{i \geq 0} H^{i}(\Omega^{a+i}_X), \qquad \HH^2(X) \iso \bigoplus_{i \geq 0} H^i(\Lambda^{2-i}(T_X))
  \]
  the Hochschild-homological action of $H^1(T_X) \subset \HH^2(X)$ on $\HH_\bullet(X)$ differs from the one induced from the contraction morphism $T_X \otimes \Omega^1_X \to \O_X$ only by the multiplication with the Todd class. Since in~\eqref{eq:action on hochschild cohomology} we have $a \neq 0$ and $a+2N \neq 0$, we are only interested in what happens in the middle cohomology of the hypersurface, $H^n(X)$, and thus the twist by the Todd class does not matter for our purposes since it only changes the result by corrections in other cohomological degrees. Hence it is enough to find two integers,~$p < q$, none of which is equal to~$n/2$, and a complementary class~$\widetilde{\xi} \in H^1(X, T_X)^{\otimes q-p-1}$ such that the multiplication map
  \begin{equation}
    \label{eq:polyvector multiplication}
    H^{q, n-q}(X) \xrightarrow{\xi \cdot -} H^{q-1, n-q+1}(X) \xrightarrow{\widetilde{\xi} \cdot -} H^{p, n-p}(X)
  \end{equation}
  is nonzero.

  By Lemma~\ref{lem:nondiagonal diamonds for hypersurfaces} the Hodge diamond of $X$ is not diagonal. By symmetry there are at least two integers $p < q$, none of which are equal to $n/2$, such that $H^{p, n-p}(X)$ and $H^{q, n-q}(X)$ are both nonzero. A refinement of Griffiths' theorem (see, e.g., \cite[Th.~2.2]{donagi}) shows that not only those cohomology groups are isomorphic to the components of the Jacobian ring, but also the action of $H^1(X, T_X) \iso J^d(f)$ is given by the multiplication in the ring. The multiplication in the Jacobian ring is nondegenerate (see, e.g, \cite[Th.~2.6]{donagi}), and hence it is possible to choose a complementary class $\widetilde{\xi}$ such that the multipilcation~\eqref{eq:polyvector multiplication} is nonzero. Thus the injectivity is proved.
\end{proof}

We are now ready to prove the main theorem of this paper.

\begin{theorem}[{ = Theorem~\ref{thm:main theorem intro}}]
  \label{thm:main theorem}
  Let $X \subset \P(V)$ be a smooth $n$-dimensional hypersurface of degree~$d < n+2$ given by the equation~$f \in S^d V^\dual$. Assume that~$d \geq 4$, or~$d \geq 3$ and~$n > 3$. Let~$\mA_X \subset \Dbcoh(X)$ be the Kuznetsov component of $X$ (Definition~\textup{\ref{def:kuznetsov component}}), and let $\Phi_{\mA_X}$ be the rotation functor (Definition~\textup{\ref{def:rotation functor}}).
  Then the pair $(\mA_X, \Phi_{\mA_X})$, as a dg-category with a dg-endofunctor, determines $X$ up to an isomorphism.
\end{theorem}
\begin{proof}
  For any $k \geq 0$ the vector space of dg-natural transformations $\mathrm{id}_{\mA_X} \Rightarrow \Phi^{\circ k}$, i.e., the homotopy classes of maps in the dg-category of functors from $\mA_X$ to itself, is naturally isomorphic to $\Hom_{X \times X}(P_0, P_k)$ \cite{toen}. By Lemma~\ref{lem:multiplication on kuznetsov component} the vector space $\Hom(P_0, P_1)$ is isomorphic to $V^\dual$. The composition of the maps defines for any $k \geq 0$ a morphism
  \[
    (V^\dual)^{\otimes k} \to \Hom(P_0, P_k),
  \]
  and it factors through the map $m_{P_k}\colon S^k V^\dual \to \Hom(P_0, P_k)$ defined in Lemma~\ref{lem:multiplication on kuznetsov component} by construction. Thus for $k = d$ we get a commutative diagram:

  \[\begin{tikzcd}
      {S^d V^\dual} & {\mathrm{dgNat}(\mathrm{id}_{\mA_X}, \Phi_{\mA_X}^{\circ d})} \\
      & {\Hom_{X \times X}(P_0, P_d)}
      \arrow[from=1-1, to=1-2]
      \arrow["{m_{P_d}}"', from=1-1, to=2-2]
      \arrow["\caniso", from=1-2, to=2-2]
    \end{tikzcd}\]

  By Theorem~\ref{thm:kuznetsov periodicity} the vector space $\Hom(P_0, P_d)$ is isomorphic to $\Ext^2(P_0, P_0)$, and the kernel of the diagonal arrow is equal to the kernel of the composition
  \[
    S^d V^\dual \to H^1(X, T_X) \monoarrow \HH^2(X) \to \HH^2(\mA_X)
  \]
  By Lemma~\ref{lem:action is nondegenerate if non-diagonal} the composition of the last two arrow is injective. By Proposition~\ref{prop:hh2 and jacobian ideal} the kernel of the first morphism is equal to the $d$'th component of the Jacobian ideal of $f$.

  Thus, up to an automorphism of $V$, we reconstructed the $d$'th component of the Jacobian ideal of $f$ as a subspace in $S^d V^\dual$ from the pair $(\mA_X, \Phi)$. This subspace, in turn, recovers the hypersurface $X$ up to an automorphism of $V$ by Mather--Yau theorem \cite[Prop.~1.1]{donagi}.
\end{proof}

\begin{corollary}
  Let $X \subset \P(V)$ be an $n$-dimensional hypersurface of degree $n+1$. Assume that $n \geq 3$. Then the Kuznetsov component $\mA_X \subset \Dbcoh(X)$, considered as a dg-category, determines $X$ up to an isomorphism.
\end{corollary}
\begin{proof}
  If $d = n+1$, the endofunctor $\Phi$ defined in the statement of Theorem~\ref{thm:main theorem} is, up to a shift, the inverse Serre functor of $\mA_X$ \cite[Lem.~4.1]{kuznetsov-v14}. Thus it can be canonically recovered as a dg-endofunctor from the dg-structure on $\mA_X$, and the result follows from Theorem~\ref{thm:main theorem}.
\end{proof}

\printbibliography

\end{document}